\documentclass[reqno]{amsart}
\usepackage{amsmath}
\usepackage{amssymb}
\usepackage{amsfonts}
\usepackage{mathtools}
\usepackage[ruled,vlined,linesnumbered,norelsize,algosection]{algorithm2e}
\usepackage{float}
\usepackage{enumitem}
\usepackage{chngcntr}
\usepackage{url}

\floatstyle{ruled}
\newfloat{myalgorithm}{thp}{lop}
\floatname{myalgorithm}{Algorithm}

\counterwithin{myalgorithm}{section}

\newtheorem{thm}{Theorem}
\newtheorem{prop}[thm]{Proposition}
\newtheorem{lem}[thm]{Lemma}
\theoremstyle{remark}
\newtheorem{rem}[thm]{Remark}

\numberwithin{equation}{section}
\numberwithin{thm}{section}

\newcommand{\ZZ}{\mathbb{Z}}
\DeclareMathOperator{\FF}{\mathsf{F}}
\DeclareMathOperator{\MM}{\mathsf{M}}

\renewcommand{\leq}{\leqslant}
\renewcommand{\geq}{\geqslant}
\newcommand{\defeq}{\coloneqq}
\DeclareMathOperator{\ord}{ord}

\begin{document}

\title[Exponent one-fifth integer factorisation]{An exponent one-fifth algorithm for deterministic integer factorisation}
\author{David Harvey}
\email{d.harvey@unsw.edu.au}
\address{School of Mathematics and Statistics, University of New South Wales, Sydney NSW 2052, Australia}

\begin{abstract}
Hittmeir recently presented a deterministic algorithm that provably computes
the prime factorisation of a positive integer~$N$
in $N^{2/9+o(1)}$ bit operations.
Prior to this breakthrough, the best known complexity bound
for this problem was $N^{1/4+o(1)}$,
a result going back to the 1970s.
In this paper we push Hittmeir's techniques further,
obtaining a rigorous, deterministic factoring algorithm
with complexity $N^{1/5+o(1)}$.
\end{abstract}

\maketitle

\section{Introduction}
\label{sec:intro}

Let $\FF(N)$ denote the time required to compute the prime factorisation
of an integer $N \geq 2$.
By ``time'' we mean ``number of bit operations'',
or more precisely, the number of steps performed by a deterministic
Turing machine with a fixed, finite number of linear tapes
\cite{Pap-complexity}.
All integers are assumed to be encoded in the usual binary representation.

In this paper we prove the following result:
\begin{thm}
\label{thm:main}
  There is an integer factorisation algorithm achieving
    \[ \FF(N) = O(N^{1/5} \log^{16/5} N). \]
\end{thm}
One may show that the space complexity of this algorithm is
$O(N^{1/5} \log^{11/5} N)$,
but we will not give the details of this analysis.

We emphasise that the new algorithm is \emph{deterministic},
the complexity bound is \emph{rigorously established},
and the algorithm runs on a \emph{classical computer}.
If one is willing to relax these conditions,
better complexity bounds are known:
\begin{itemize}[label=--,labelwidth=0.7cm,leftmargin=\labelwidth,itemsep=0.2em,topsep=\itemsep]
  \item
  The class group relations method of Lenstra and Pomerance \cite{LP-rigorous}
  provably runs in expected time
    \[ \exp\big((1 + o(1)) (\log N)^{1/2} (\log \log N)^{1/2}\big). \]
  This is the best known rigorous bound for a probabilistic factoring algorithm.
  
  \item
  The general number field sieve is conjectured to run in time
    \[ \exp\big(((64/9)^{1/3} + o(1)) (\log N)^{1/3} (\log \log N)^{2/3}\big). \]
  See \cite[\S6.2]{CP-primes} for a discussion of the heuristics
  supporting this conjecture.
  This algorithm is responsible for all recent record-breaking factorisations,
  such as the factorisation of the 250-digit challenge number RSA-250
  \cite{BGGHTZ-RSA250}.
  
  \item Shor's factoring algorithm runs in polynomial time,
  specifically $(\log N)^{2+o(1)}$, on a quantum computer
  \cite{Sho-factoring}.
\end{itemize}

It is also worth mentioning the deterministic algorithm of
Shanks \cite{Sha-classnumber},
which achieves $\FF(N) = N^{1/5+o(1)}$,
conditional on the (still unproved) Generalised Riemann Hypothesis.
According to \cite[\S5.6.4]{CP-primes},
the complexity of this algorithm is actually $N^{1/5} \log^{4+o(1)} N$,
which is slightly worse than our bound in Theorem \ref{thm:main}.
  
The literature on integer factorisation is vast.
For further pointers, see \cite{CP-primes}.
For the rest of the paper, we focus on algorithms that are
deterministic, rigorous, and classical.

The naive trial division method,
i.e., checking all possible divisors up to $N^{1/2}$,
runs in time $N^{1/2+o(1)}$.
The first algorithm faster than this was that of Lehman \cite{Leh-factoring},
who obtained $\FF(N) = N^{1/3+o(1)}$.
Lehman's main theorem shows that if $N$ is a semiprime,
i.e., $N = pq$ for distinct primes $p$ and $q$,
then there exists a pair of ``small'' integers $a$ and $b$ such that
$aq + bp$ lies in a certain short interval (depending on $a$ and $b$).
His algorithm examines every integer in this interval,
for every $a$ and $b$ in a suitable range,
and succeeds in factoring $N$ as soon as it encounters $aq + bp$.

Shortly after Lehman's work appeared,
Pollard \cite{Pol-factorization} and Strassen \cite{Str-factoring}
independently published two closely related algorithms achieving
$\FF(N) = N^{1/4+o(1)}$.
Both of these algorithms depend heavily on the FFT
(fast Fourier transform).
Strassen's method was revisited by Bostan, Gaudry and Schost
\cite{BGS-recurrences},
and then again by Costa and the present author \cite{CH-factor},
who achieved
\begin{equation}
\label{eq:CH-bound}
  \FF(N) = O\left(\frac{N^{1/4} \log^2 N}{(\log \log N)^{1/2}}\right).
\end{equation}
(The results in \cite{BGS-recurrences} and \cite{CH-factor}
were actually stated in terms of the function $\MM(n)$,
the cost of multiplying $n$-bit integers.
It is now known that $\MM(n) = O(n \log n)$ \cite{HvdH-nlogn},
and at the risk of anachronism,
we will use this result in our historical discussion to simplify
several complexity statements from the literature.)

Another family of algorithms achieving $\FF(N) = N^{1/4+o(1)}$
is based on Coppersmith's method \cite{Cop-lowexponent}
for finding small roots of polynomial equations modulo~$N$
(see also \cite[\S5]{Ber-smallheight}).
These algorithms rely on fast lattice reduction techniques such as LLL.
Their complexity is presumably of the form $O(N^{1/4} \log^k N)$
for some small positive integer $k$,
but the author is not aware of a precise statement in the literature.
This approach has the advantage that the space complexity is only
$(\log N)^{O(1)}$,
comparing favourably to the $N^{1/4+o(1)}$ space complexity of the
Pollard--Strassen method.

In 2016, Hittmeir established the bound
\begin{equation}
\label{eq:hittmeir-bound1}
  \FF(N) = O\left(N^{1/4} \exp\left(\frac{-C_1 \log N}{\log (C_2 \log N)}\right) \log^2 N\right),
\end{equation}
for certain explicit constants $C_1, C_2 > 0$ \cite{Hit-BSGS}.
This bound is still $N^{1/4+o(1)}$,
but improves on the ``optimised Strassen bound'' \eqref{eq:CH-bound} by a superpolynomial factor,
i.e., by an arbitrarily large power of $\log N$.
Hittmeir's algorithm relies on an trick that seems to have first been suggested
by Pollard \cite[\S3]{Pol-montecarlo}:
if $N$ is a semiprime $pq$, then Euler's theorem implies that
\begin{equation}
\label{eq:hittmeir-congruence}
  \alpha^{p+q} \equiv \alpha^{N+1} \pmod N
\end{equation}
for any $\alpha \in \ZZ_N^*$.
One may therefore try to compute $p+q$ (and hence factor~$N$)
by solving $\alpha^x \equiv \alpha^{N+1} \pmod N$,
using for instance Shanks' BSGS (baby-step/giant-step) algorithm
for computing discrete logarithms~\cite{Sha-classnumber}.
McKee and Pinch~\cite[\S2]{MP-oldandnew} turn this idea into a deterministic
factoring algorithm with $\FF(N) = N^{1/3+o(1)}$,
but Hittmeir squeezes much more out of it.
Among other things, he shows that the relation $pq = N$ forces certain
restrictions on the residue of $p+q$ modulo various small primes,
which can be used to speed up the BSGS search.

More recently, in a spectacular preprint \cite{Hit-timespace},
Hittmeir presented an algorithm achieving
  \[ \FF(N) = O(N^{2/9} \log^{11/3} N). \]
Notice that $2/9 = 0.222\ldots < 0.25 = 1/4$.
This result represents the first exponential speedup for the problem of
integer factorisation since the 1970s.
The algorithm has many features in common with \cite{Hit-BSGS},
but instead of searching for $p+q$,
it searches for the linear combinations $aq + bp$
that arise in Lehman's algorithm.
The replacement for \eqref{eq:hittmeir-congruence} in this context
is the congruence
  \[ \alpha^{aq + bp} \equiv \alpha^{aN + b} \pmod p, \]
which follows from Fermat's little theorem.
His algorithm actually examines many more candidates than Lehman's algorithm,
but the complexity is drastically reduced thanks to the use of the BSGS search.

The algorithm described in the present paper achieves a further exponential
speedup, from $N^{2/9+o(1)}$ to $N^{1/5+o(1)} = N^{0.2+o(1)}$.
It is based largely on Hittmeir's algorithm from \cite{Hit-timespace}.
The main difference, and the source of the exponential speedup,
is that we apply the BSGS search technique much more aggressively.
Whereas Hittmeir breaks up the search space into chunks and
applies the BSGS search to each chunk separately,
we show how to sweep through the entire space with a single BSGS search.
For a more quantitative discussion of the
relationship between the two algorithms,
see Remark \ref{rem:discussion}.

We also make one other small improvement to Hittmeir's algorithm:
we replace a certain multipoint evaluation step with an invocation of
Bluestein's algorithm \cite{Blu-dft} (see Algorithm \ref{algo:collisions}).
This saves a further factor of $(\log N)^{O(1)}$ overall.

\section{Preliminaries}

In this section we introduce various notation and terminology,
and recall several ingredients from the literature that will be needed later.
All results stated in this section without specific references may be found
in standard textbooks on computer algebra, such as
\cite{vzGG-compalg3} or \cite{BZ-mca}.

For $n$ a positive integer, we define
$\lg n \defeq \lceil \log n / \log 2 \rceil$ if $n \geq 2$,
and $\lg 1 \defeq 1$.
Clearly $\lg n = \Theta(\log n)$ as $n \to \infty$.

Given a list of $n \geq 1$ elements of bit size $\beta \geq 1$,
we may sort the list using the ``merge sort'' algorithm
in time $O(n \beta \lg n)$ \cite{Knu-TAOCP3}.
Here we assume that the elements belong to some totally ordered set,
and that we may compare a pair of elements in time $O(\beta)$.

\subsection{Integer arithmetic}

Let $n \geq 1$, and assume that we are given integers $x$ and $y$ such that
$|x|, |y| \leq 2^n$.
\begin{itemize}
  \item
  We may compute $x + y$ and $x - y$ in time $O(n)$.

  \item
  We write $\MM(n)$ for the cost of computing the product $xy$.
  As mentioned in the Introduction, we may take
  $\MM(n) = O(n \lg n)$~\cite{HvdH-nlogn}.

  \item 
  If $y > 0$, we may compute the quotients $\lfloor x/y \rfloor$ and
  $\lceil x/y \rceil$ in time $O(\MM(n))$.
  We may therefore also compute the residue of $x$ modulo $y$
  in the interval $[0, y)$ in time $O(\MM(n))$.
  
  \item
  More generally, for a fixed positive rational number $u/v$,
  and assuming that $x, y > 0$, we may compute
  $\lfloor (x/y)^{u/v} \rfloor$ and $\lceil (x/y)^{u/v} \rceil$
  in time $O(\MM(n))$.
  
  \item
  Using the ``half-GCD'' algorithm,
  we may compute $g \defeq \gcd(x, y)$,
  and if desired find integers $u$ and $v$ such that $ux + vy = g$,
  in time $O(\MM(n) \lg n) = O(n \lg^2 n)$.
  In particular, if $\gcd(x, y) = 1$,
  we may compute the inverse of $x$ modulo $y$
  in time $O(\MM(n) \lg n)$.
\end{itemize}

Let $N \geq 2$.
We write $\ZZ_N$ for the ring $\ZZ/N\ZZ$ of integers modulo $N$.
Elements of $\ZZ_N$ will always be represented by their residues in the
interval $[0, N)$.
In particular, an element of $\ZZ_N$ occupies at most $\lg N$ bits of storage.
We write $\ZZ_N^*$ for the group of units of $\ZZ_N$,
i.e., the subset of those $x \in \ZZ_N$ for which $\gcd(x, N) = 1$.

\begin{itemize}
  \item
  Given $x, y \in \ZZ_N$,
  we may compute $x \pm y \in \ZZ_N$ in time $O(\lg N)$,
  and we may compute $xy \in \ZZ_N$ in time $O(\MM(\lg N)) = O(\lg N \lg \lg N)$.

  \item
  Given $x \in \ZZ_N$ and an integer $m \geq 0$,
  we may compute $x^m \in \ZZ_N$ in time $O(\MM(\lg N) \lg m)$,
  using the ``repeated squaring'' algorithm.

  \item
  Given $x \in \ZZ_N$, we may test whether $x \in \ZZ_N^*$
  in time $O(\MM(\lg N) \lg \lg N)$, by computing $\gcd(x, N)$.
  If this GCD is not $1$ or $N$,
  and if $N = pq$ is a semiprime,
  then we immediately recover $p$ and $q$ in time $O(\MM(\lg N))$.
\end{itemize}

\subsection{Polynomial arithmetic}

We always represent a polynomial in $\ZZ_N[x]$ by a list of its coefficients.
For $d \geq 1$, let $\MM_N(d)$ denote the cost of multiplying two polynomials
in $\ZZ_N[x]$ of degree at most $d$.
\begin{lem}
\label{lem:polynomial-bound}
  If $d = O(N)$ then
  \begin{equation}
  \label{eq:polynomial-bound}
  \MM_N(d) = O(d \lg^2 N).
  \end{equation}
\end{lem}
\begin{proof}
Using Kronecker substitution, i.e., packing the coefficients into large
integers and then multiplying these integers
(see \cite[Cor.~8.27]{vzGG-compalg3}), we obtain
  \[ \MM_N(d) = O(\MM(d (\lg N + \lg d))). \]
Under the hypothesis $d = O(N)$ this simplifies to $O(d \lg^2 N)$.
\end{proof}
\begin{rem}
  In the above lemma,
  the hypothesis ``$d = O(N)$'' means that for every constant $C > 0$,
  the lemma holds for all $d$ and $N$ in the region $d < CN$,
  where the implied big-$O$ constant in \eqref{eq:polynomial-bound}
  may depend on $C$.
  For the rest of the paper, a similar remark applies whenever we use
  the big-$O$ notation in this way.
\end{rem}

\begin{lem}
\label{lem:product-tree}
  Let $n \geq 1$ with $n = O(N)$.
  Given as input $v_1, \ldots, v_n \in \ZZ_N$,
  we may compute the polynomial
  $f(x) = (x - v_1) \cdots (x - v_n) \in \ZZ_N[x]$
  in time
    \[ O(n \lg^3 N). \]
\end{lem}
\begin{proof}
  This is a standard application of a product tree:
  starting with the linear polynomials $x - v_i$,
  we repeatedly multiply pairs of polynomials together,
  with the degree doubling at each level.
  There are $O(\lg n) = O(\lg N)$ levels,
  and Lemma \ref{lem:polynomial-bound} implies that
  the cost at each level is $O(n \lg^2 N)$.
  For further details, see for instance \cite[Algorithm 10.3]{vzGG-compalg3}.
\end{proof}

\begin{lem}
\label{lem:bluestein}
  Given as input an element $\alpha \in \ZZ^*_N$,
  positive integers $m, n = O(N)$,
  and $f \in \ZZ_N[x]$ of degree $n$,
  we may compute $f(1), f(\alpha), \ldots, f(\alpha^{m-1}) \in \ZZ_N$
  in time
    \[ O((n + m) \lg^2 N). \]
\end{lem}
\begin{proof}
  We use a variant of Bluestein's trick \cite{Blu-dft}
  (following \cite[Lem.~3]{HvdH-zmatmult}).
  Let $f(x) = \sum_{j=0}^n f_j x^j$.
  Then the identity $ij = \binom{i}{2} + \binom{-j}{2} - \binom{i-j}{2}$
  implies that
  \begin{equation}
  \label{eq:eval-f}
    f(\alpha^i) = \sum_{j=0}^n f_j \alpha^{ij} = h_i \sum_{j=0}^n f'_j g_{i-j},
  \end{equation}
  where
    \[ h_i  \defeq \alpha^{\binom{i}{2}}, \qquad
       f'_j \defeq \alpha^{\binom{-j}{2}} f_j, \qquad
       g_k  \defeq \alpha^{-\binom{k}{2}}. \]
  Using the identities $\binom{i+1}{2} = \binom{i}{2} + i$ and
  $\binom{-j}{2} = \binom{j+1}{2}$,
  we may compute $h_0, \ldots, h_{m-1}$, $f'_0, \ldots, f'_n$ and
  $g_{-n}, \ldots, g_{m-1}$ from $\alpha$, $\alpha^{-1}$
  and the coefficients $f_j$ using $O(n + m)$ multiplications in $\ZZ_N$.
  For each $i \in \{0, \ldots, m-1\}$,
  the sum $\sum_{j=0}^n f'_j g_{i-j}$ in \eqref{eq:eval-f}
  is equal to the coefficient of $x^i$ in
  the product of the Laurent polynomials
    \[ f' \defeq \sum_{j=0}^n f'_j x^j \qquad \text{and} \qquad
       g  \defeq \sum_{k=-n}^{m-1} g_k x^k.  \]
  The problem thus reduces to multiplying a polynomial of degree $n$
  by a polynomial of degree $n + m - 1$,
  plus $O(n + m)$ auxiliary multiplications in $\ZZ_N$.
  The claimed complexity bound then follows from
  Lemma \ref{lem:polynomial-bound}.
\end{proof}

\subsection{Strassen's method}

The following result,
which goes back to Strassen \cite{Str-factoring},
may be used to quickly rule out ``small'' factors of a given integer $N$.
\begin{prop}
\label{prop:small-factors}
  Given as input a positive integer $M = O(N)$,
  we may test if $N$ has a prime divisor $p \leq M$,
  and if so find the smallest such divisor,
  in time
    \[ O(M^{1/2} \lg^3 N). \]
\end{prop}
\begin{proof}
  We give only a sketch; for details (and improvements) see
  \cite{Str-factoring}, \cite{BGS-recurrences} or \cite{CH-factor}.
  Let $d \defeq \lceil M^{1/2} \rceil$
  and compute $f(x) = (x+1) \cdots (x+d) \in \ZZ_N[x]$
  using Lemma \ref{lem:product-tree}.
  Then evaluate $f(x)$ simultaneously at $x = 0, d, \ldots, (d-1)d$,
  using a fast multipoint evaluation algorithm.
  If any $f(jd) = (jd+1)\cdots(jd+d)$ has non-trivial GCD with $N$,
  then one of its constituents $jd + i$ must have non-trivial GCD with $N$,
  and this observation may be used to locate any prime divisors of $N$
  such that $p \leq d^2$
  (using a similar procedure to the closely related
  Algorithm \ref{algo:collisions} of the present paper).
\end{proof}
\begin{rem}
  Taking $M \defeq \lfloor N^{1/2} \rfloor$ in the above result leads
  directly to Strassen's $N^{1/4+o(1)}$ integer factorisation algorithm.
\end{rem}

\subsection{Finding elements of large order}

The following result,
which allows us to find elements of large order in $\ZZ_N^*$,
is due to Hittmeir.
Of course in practice this is easy, because almost every element of $\ZZ_N^*$
has large order.
The point of Hittmeir's result is that the algorithm is deterministic
and gives provably correct output.
\begin{prop}[\protect{\cite[Theorem~6.3]{Hit-BSGS}}]
\label{prop:find-alpha}
  There is an algorithm with the following properties.
  It takes as input integers $N \geq 2$ and $D$ such that
  $N^{2/5} \leq D \leq N$.
  It returns either some $\alpha \in \ZZ_N^*$
  such that $\ord_N(\alpha) > D$,
  or a nontrivial factor of $N$, or ``$N$ is prime''.
  Its running time is
    \[ O(D^{1/2} \lg^2 N). \]
\end{prop}
The algorithm is somewhat involved, and we give only a brief summary here.
The idea is to try to compute the orders in $\ZZ_N$ of candidates
$\alpha = 2, 3$, and so on, using a BSGS search up to some bound.
If for some candidate $\alpha$ we fail to compute the order,
then the order must be large and we are done.
On the other hand, if we succeed in computing the order of a candidate,
then this reveals valuable information about the possible prime divisors of $N$.
Continuing in this way, Hittmeir shows that eventually we recover enough
information about these divisors that it becomes feasible
to search for them directly using a variant of the algorithm of
Proposition \ref{prop:small-factors}.
For details, see \cite{Hit-BSGS}.

\begin{rem}
  According to \cite[Remark~6.4]{Hit-BSGS},
  the assumption ${D \geq N^{2/5}}$ is not the sharpest possible;
  the same result may be proved under the weaker hypothesis
  $D = \Omega(N^{1/3+o(1)})$ for suitable $o(1)$.
  Luckily for us,
  the $N^{2/5}$ bound as stated is good enough for our application
  (but only just).
\end{rem}

\section{Lehman's strategy}

Given an integer $N$ that is either prime or a semiprime $pq$,
the main algorithm in Section \ref{sec:main} will attempt to factor $N$
by searching for integers of the form $aq + bp$ for various small values
of $a$ and $b$.
Before considering Lehman's theorem in detail,
we show (following \cite[Lem.~4.1]{Hit-timespace}) that it is easy
to recover $p$ and $q$ once we have guessed the correct value of $aq + bp$.
\begin{lem}
\label{lem:recover}
  Given as input an integer $N \geq 2$ such that
  $N$ is either a prime or a semiprime $pq$,
  and positive integers $a, b, u = O(N)$,
  we may test if $u$ is of the form $aq + bp$,
  and if so recover $p$ and $q$,
  in time $O(\MM(\lg N))$.
\end{lem}
\begin{proof}
  If $N = pq$ and $u = aq + bp$, then $aq$ and $bp$ are roots of the
  polynomial $Q(y) \defeq y^2 - uy + abN$.
  We first check if $Q(y)$ has rational roots,
  by testing whether $u^2 - 4abN$ is a square,
  and if so, compute the roots via the quadratic formula.
  It is then straightforward to recover $p$ and $q$, if they exist.
\end{proof}
\begin{rem}
\label{rem:lehman}
  In fact, as is implicit in Lehman's paper,
  we may still recover $p$ and $q$ without knowing $a$ and $b$,
  provided that we know the product $ab$ and that $\gcd(ab, N) = 1$.
  This requires one additional GCD operation,
  so the cost rises to $O(\MM(\lg N) \lg \lg N)$.
\end{rem}

Now we state a result that is almost equivalent to Lehman's main theorem
in \cite{Leh-factoring}.
Actually, Lehman's original statement would be perfectly adequate for our
purposes,
but we have chosen to present a slightly different version,
together with a considerably simpler proof
(inspired by the presentation of Lehman's algorithm
in \cite[Thm.~5.1.3]{CP-primes}).
\begin{lem}
\label{lem:lehman}
  Let $p, q$ and $r$ be positive integers, and let $N = pq$.
  Assume that
  \begin{equation}
  \label{eq:p-condition}
    (N/r)^{1/2} \leq p < N^{1/2}.
  \end{equation}
  Then there exist positive integers $a$ and $b$ with $1 \leq ab \leq r$ such that
  \begin{equation}
  \label{eq:lehman-inequality}
    0 \leq aq + bp - (4abN)^{1/2} < \frac{N^{1/2}}{4 r (ab)^{1/2}}.
  \end{equation}
\end{lem}
\begin{proof}
  Theorem 36 of \cite{HW-intro-NT} states that for any real number
  $\xi \in (0,1)$ and any integer $n \geq 1$,
  there exist integers $a$ and $b$,
  with $1 \leq b \leq n$ and $0 \leq a \leq b$, such that
    \[ \left| \frac{a}{b} - \xi \right| \leq \frac{1}{(n+1)b}. \]
  We apply this result for
    \[ \xi \defeq p/q, \qquad n \defeq \lfloor B \rfloor, \qquad B \defeq (qr/p)^{1/2} > 1. \]
  Since $n \leq B < n+1$, we obtain
    \[ \left| \frac{a}{b} - \frac{p}{q} \right| < \frac{1}{Bb}. \]
  The hypothesis $p \geq (N/r)^{1/2} = (pq/r)^{1/2}$ yields
  $p/q \geq (p/qr)^{1/2} = 1/B$. Therefore
    \[ \frac{a}{b} > \frac{p}{q} - \frac{1}{Bb} \geq \frac{1}{B} - \frac{1}{B} = 0, \]
  showing that $a$ must be positive.
  Next, the estimate
    \[ ab = \frac{a}{b} b^2 < \left(\frac{p}{q} + \frac{1}{Bb}\right) b^2
          = \frac{p}{q} b^2 + \frac{b}{B} \leq \frac{p}{q} B^2 + 1 = r + 1 \]
  implies that $ab \leq r$.
  Finally, the inequality of arithmetic and geometric means yields
  $aq + bp \geq (4abN)^{1/2}$, so
    \[ 0 \leq aq + bp - (4abN)^{1/2} = \frac{(aq + bp)^2 - 4abN}{aq + bp + (4abN)^{1/2}}
      \leq \frac{(aq - bp)^2}{4(abN)^{1/2}}. \]
  Since $|aq - bp| < q/B = (N/r)^{1/2}$, we obtain
    \[ 0 \leq aq + bp - (4abN)^{1/2} < \frac{N/r}{4(abN)^{1/2}} = \frac{N^{1/2}}{4r(ab)^{1/2}}. \qedhere \]
\end{proof}

\begin{rem}
\label{rem:discussion}
  Lemma \ref{lem:lehman} shows that the total number of
  candidates for $aq + bp$,
  i.e., the number of integers $x$ such that
    \[ 0 \leq x - (4kN)^{1/2} < \frac{N^{1/2}}{4rk^{1/2}} \]
  for some $k$ in the range $1 \leq k \leq r$, is at most
    \[ \sum_{k=1}^r \left\lceil \frac{N^{1/2}}{4rk^{1/2}} \right\rceil
       \leq \sum_{k=1}^r \left(\frac{N^{1/2}}{4rk^{1/2}} + 1\right)
      = O\left( \frac{N^{1/2}}{r^{1/2}} + r \right). \]
  Lehman optimises this last expression by taking
  $r \asymp N^{1/3} = N^{0.333\ldots}$,
  so that the number of candidates is $O(N^{1/3})$.
  He checks each one individually,
  using the procedure mentioned in Remark \ref{rem:lehman},
  leading to his overall $N^{1/3+o(1)}$ bound.

  Hittmeir \cite{Hit-timespace} takes a smaller value of $r$,
  namely $r = N^{2/9+o(1)} = N^{0.222\ldots+o(1)}$,
  so that the number of candidates increases to
  $N^{7/18+o(1)} = N^{0.388\ldots+o(1)}$.
  (Note that in \cite{Hit-timespace}, the parameter $r$ is called $\eta$.)
  Taking a smaller value of $r$ introduces more structure
  into the space of candidates;
  Hittmeir takes advantage of this by using BSGS techniques
  to speed up the search instead of examining each candidate individually.
  First, for each $k$ up to $N^{1/9+o(1)} = N^{0.111\ldots+o(1)}$
  (Hittmeir calls this threshold $\xi$),
  he runs a BSGS search to examine all candidates corresponding
  to that one value of $k$.
  Then he performs one more BSGS search to examine all
  remaining values of $k$ up to
  $r = N^{2/9+o(1)} = N^{0.222\ldots+o(1)}$.
  Altogether, this leads to an overall complexity of $N^{2/9+o(1)}$.

  Our strategy (see Section \ref{sec:main})
  is to take $r$ smaller again,
  namely $r = N^{1/5+o(1)} = N^{0.2+o(1)}$.
  The number of candidates increases even further to
  $N^{2/5+o(1)} = N^{0.4+o(1)}$.
  However, we perform a single BSGS search
  (similar to the method Hittmeir uses for the large values of $k$)
  that covers all values of $k$ simultaneously.
  In effect, we enjoy a ``square-root speedup'' over the entire search space.

  An interesting question is whether it is possible to obtain a fully
  square-root speedup for Lehman's original choice $r \asymp N^{1/3}$.
  This would presumably lead to a factoring algorithm with complexity
  $N^{1/6+o(1)}$.
\end{rem}

\section{The main algorithm}
\label{sec:main}

Before discussing the main search algorithm,
we first consider Algorithm \ref{algo:collisions},
which finds collisions modulo $p$ or $q$ between two subsets of $\ZZ_N$.
This subroutine is almost identical to
Algorithm 3.2 of \cite{Hit-timespace}.
However, whereas Hittmeir allows both subsets to be essentially arbitrary,
we specialise to the case that one of the subsets consists entirely
of powers of some fixed $\alpha \in \ZZ^*_N$.
This allows us to improve the complexity from $O((n + m) \lg^3 N)$
(see \cite[Lem.~3.3]{Hit-timespace}) to $O(n \lg^3 N + m \lg^2 N)$.
\begin{myalgorithm}[ht]
\caption{-- finding collisions modulo $p$ or $q$}
\label{algo:collisions}
  \begin{flushleft}
  \textbf{Input:}
  \begin{itemize}[label=--,labelwidth=0.7cm,leftmargin=\labelwidth]
    \item A positive integer $N$, either prime or semiprime.
    \item A positive integer $m$,
      and an element $\alpha \in \ZZ_N^*$ with $\ord_N(\alpha) \geq m$.
    \item Elements $v_1, \ldots, v_n \in \ZZ_N$
      for some positive integer $n$,
      such that $v_h \neq \alpha^i$
      for all $h \in \{1, \ldots, n\}$ and $i \in \{0, \ldots, m-1\}$.
  \end{itemize}
  \textbf{Output:}
  \begin{itemize}[label=--,labelwidth=0.7cm,leftmargin=\labelwidth]
    \item If $N = pq$ is a semiprime,
    and if there exists $h \in \{1, \ldots, n\}$ such that
    $v_h \equiv \alpha^i \pmod p$ or $v_h \equiv \alpha^i \pmod q$
    for some $i \in \{0, \ldots, m-1\}$, returns $p$ and $q$.
    \item Otherwise returns ``no factors found''.
  \end{itemize}
  \end{flushleft}
  \vspace{2pt} \hrule height 0.2pt \vspace{1pt}
  \begin{enumerate}[label=\arabic*),labelwidth=0.7cm,leftmargin=\labelwidth,itemsep=0.2em,topsep=\itemsep]
    \item Using Lemma \ref{lem:product-tree} (product tree),
      compute the polynomial
      \[ f(x) \defeq (x - v_1) \cdots (x - v_n) \in \ZZ_N[x]. \]
    \item Using Lemma \ref{lem:bluestein} (Bluestein's algorithm),
      compute the values $f(\alpha^i) \in \ZZ_N$ for $i = 0, \ldots, m-1$.
    \item For $i = 0, \ldots, m-1$:
    \begin{enumerate}[label=\alph*),itemsep=0.2em,topsep=\itemsep]
      \item Compute $\gamma_i \defeq \gcd(N, f(\alpha^i))$.
      \item If $\gamma_i \notin \{1, N\}$, recover $p$ and $q$ and return.
      \item If $\gamma_i = N$, then for $h = 1, \ldots, n$:
	\begin{enumerate}[label=\roman*),itemsep=0.2em,topsep=\itemsep]
	  \item If $\gcd(N, v_h - \alpha^i) \neq 1$,
	  recover $p$ and $q$ and return.
	\end{enumerate}
    \end{enumerate}
    \item Return ``no factors found''.
  \end{enumerate}
\end{myalgorithm}

\begin{prop}
\label{prop:collisions}
  Algorithm \ref{algo:collisions} is correct.
  For $m, n = O(N)$, its running time is
    \[ O(n \lg^3 N + m \lg^2 N). \]
\end{prop}
\begin{proof}
  We first prove correctness.
  Suppose that $N = pq$ is a semiprime, and that $v_{h_0} \equiv \alpha^{i_0}$
  modulo $p$ or modulo $q$ for some $h_0$ and $i_0$.
  Then $f(\alpha^{i_0})$ includes the factor $v_{h_0} - \alpha^{i_0}$,
  so we have $\gamma_{i_0} \neq 1$ in Step (3a).
  If additionally $\gamma_{i_0} \neq N$,
  then $p$ and $q$ are successfully recovered in Step (3b).
  Suppose however that $\gamma_{i_0} = N$.
  Then the factors will be recovered in Step (3ci) when we reach $h = h_0$,
  as the preconditions ensure that $\gcd(N, v_{h_0} - \alpha^{i_0}) \neq N$.

  Now we analyse the complexity.
  By Lemma \ref{lem:product-tree} and Lemma \ref{lem:bluestein},
  Steps (1)--(2) run in time $O(n \lg^3 N + m \lg^2 N)$.
  Step (3a) executes at most $m$ times.
  Step (3ci) executes at most $n$ times,
  because as soon as we find $\gamma_i = N$ in Step (3c),
  we know that $(v_1 - \alpha^i) \cdots (v_n - \alpha^i) = 0$ in $\ZZ_N$,
  so one of the factors $v_h - \alpha^i$ must fail to be invertible modulo $N$
  and thus reveal the factors $p$ and $q$
  (here again we use the hypothesis that $v_h \neq \alpha^i$ in $\ZZ_N$).
  Therefore the total number of GCDs computed is at most $m + n$,
  contributing $O((m + n) \lg N \lg^2 \lg N)$ to the complexity.
\end{proof}

We now proceed to the main search algorithm (Algorithm \ref{algo:search}).
As mentioned previously, this algorithm leverages the search strategy that
Hittmeir uses in \cite{Hit-timespace} for the ``large'' values of $k$,
but we carry it out for the entire range of $k$ simultaneously.

\begin{myalgorithm}[ht]
\caption{-- the main search}
\label{algo:search}
\begin{flushleft}
\textbf{Input:}
\begin{itemize}[label=--,labelwidth=0.7cm,leftmargin=\labelwidth]
  \item Positive integers $N \geq 2$ and $r$,
  such that either $N$ is prime, or $N = pq$ is a semiprime
  satisfying \eqref{eq:p-condition},
  i.e., such that $(N/r)^{1/2} \leq p < N^{1/2}$.
  \item A positive integer $m$,
  and an element $\alpha \in \ZZ_N^*$ with $\ord_N(\alpha) \geq m$.
\end{itemize}
\textbf{Output:} If $N$ is a semiprime, returns $p$ and $q$. Otherwise returns ``$N$ is prime''.
\end{flushleft}
\vspace{2pt} \hrule height 0.2pt \vspace{1pt}
\begin{enumerate}[label=\arabic*),labelwidth=0.7cm,leftmargin=\labelwidth,itemsep=0.2em]
  \item
  Compute $\alpha^i \in \ZZ_N$ for $i = 0, \ldots, m-1$. 
  For each $i = 1, \ldots, m-1$, check whether $\gcd(N, \alpha^i - 1) = 1$.
  If not, return $p$ and $q$.
  
  \item
  For all pairs of positive integers $(a,b)$ such that $1 \leq ab \leq r$:
  \begin{enumerate}[label=\alph*),itemsep=0.2em]
    \item Compute
    \begin{equation}
    \label{eq:t-defn}
      t_{a,b} \defeq \alpha^{aN + b - \lceil (4abN)^{1/2} \rceil} \in \ZZ_N.
    \end{equation}

    \item
    For each integer $j$ in the interval
      \[ 0 \leq j < \frac{N^{1/2}}{4rm (ab)^{1/2}}, \]
    compute
    \begin{equation}
    \label{eq:v-defn}
      v_{a,b,j} \defeq \alpha^{-jm} t_{a,b} \in \ZZ_N.
    \end{equation}
  \end{enumerate}

  \item
  Applying a sort-and-match algorithm to the output of Steps (1) and (2),
  find all $i \in \{0, \ldots, m-1\}$ and all triples $(a,b,j)$
  (as in Step (2)) such that
  \begin{equation}
    \label{eq:congruence-N}
    v_{a,b,j} = \alpha^i.
  \end{equation}
  For each such match, apply Lemma~\ref{lem:recover} to the candidate
    \[ u \defeq i + jm + \lceil(4abN)^{1/2}\rceil. \]
  If Lemma \ref{lem:recover} succeeds in finding $p$ and $q$, return.
  
  \item
  Let $v_1, \ldots, v_n$ be the list of values $v_{a,b,j}$
  computed in Step (2), skipping those that were discovered in Step (3)
  to be equal to $\alpha^i$ for some $i \in \{0, \ldots, m-1\}$.
  Apply Algorithm~\ref{algo:collisions} to $v_1, \ldots, v_n$ with
  the given $N$, $\alpha$ and $m$.
  If Algorithm~\ref{algo:collisions} succeeds in finding $p$ and $q$, return.
  
  \item
  Return ``$N$ is prime''.
\end{enumerate}
\end{myalgorithm}

\begin{prop}
\label{prop:search}
  Algorithm \ref{algo:search} is correct.
  For $r, m = O(N)$, its running time is
    \[ O\left(\left(\frac{N^{1/2}}{r^{1/2} m} + r\right) \lg^4 N + m \lg^2 N \right). \]
\end{prop}
\begin{proof}
  First suppose that $N$ is a semiprime $pq$ satisfying \eqref{eq:p-condition}.
  We must prove that in Steps (1)--(4),
  the algorithm succeeds in finding $p$ and $q$.

  If $\ord_p(\alpha) < m$,
  then in Step (1) we will discover $p$ and $q$ when we compute
  $\gcd(N, \alpha^i - 1)$ for $i = \ord_p(\alpha)$,
  because the hypothesis $\ord_N(\alpha) \geq m$ ensures that
  $\alpha^i \neq 1$ for this~$i$.
  Henceforth we may assume that $\ord_p(\alpha) \geq m$.

  Applying Lemma \ref{lem:lehman}, we find that there exist positive integers
  $a_0$ and $b_0$, with $1 \leq a_0 b_0 \leq r$, such that
    \[ 0 \leq y_0 < \frac{N^{1/2}}{4r(a_0 b_0)^{1/2}}, \]
  where
    \[ y_0 = u_0 - \lceil (4 a_0 b_0 N)^{1/2} \rceil, \qquad u_0 = a_0 q + b_0 p. \]
  Observe that
    \[ \alpha^{u_0} = \alpha^{a_0 q + b_0 p} \equiv \alpha^{a_0 pq + b_0} \pmod p, \]
  since $p \equiv 1 \pmod{p-1}$ and $\alpha^{p-1} \equiv 1 \pmod p$.
  Thus
    \[ \alpha^{y_0} \equiv \alpha^{a_0 N + b_0 - \lceil (4 a_0 b_0 N)^{1/2} \rceil} \equiv t_{a_0,b_0} \pmod p, \]
  where $t_{a,b}$ is defined as in \eqref{eq:t-defn}.
  If we now decompose $y_0$ as
    \[ y_0 = i_0 + j_0 m, \qquad 0 \leq i_0 < m, \qquad 0 \leq j_0 < \frac{N^{1/2}}{4rm(a_0 b_0)^{1/2}}, \]
  we find that $\alpha^{i_0 + j_0 m} \equiv t_{a_0,b_0} \pmod p$, or in other words
  \begin{equation}
    \label{eq:congruence-p}
    v_{a_0,b_0,j_0} \equiv \alpha^{i_0} \pmod p,
  \end{equation}
  where $v_{a,b}$ is defined as in \eqref{eq:v-defn}.

  Step (3) attempts to solve \eqref{eq:congruence-p} by finding all solutions
  to the stronger congruence \eqref{eq:congruence-N} (a congruence modulo $N$).
  Let $(i,a,b,j)$ be one of the solutions to
  \eqref{eq:congruence-N} examined in Step (3).
  If it is equal to $(i_0,a_0,b_0,j_0)$,
  then $u = i + jm + \lceil (4abN)^{1/2} \rceil$
  will be equal to
    \[ i_0 + j_0 m + \lceil (4 a_0 b_0 N)^{1/2} \rceil = u_0 = a_0 q + b_0 p, \]
  and the factors $p$ and $q$ will be discovered via Lemma \ref{lem:recover}.
  Now suppose instead that $(i,a,b,j) \neq (i_0,a_0,b_0,j_0)$.
  We claim then that in fact $(a,b,j) \neq (a_0,b_0,j_0)$.
  Indeed, if $(a,b,j) = (a_0,b_0,j_0)$ then
    \[ \alpha^{i_0} \equiv v_{a_0,b_0,j_0} \equiv v_{a,b,j} \equiv \alpha^i \pmod p, \]
  which forces $i_0 = i$ because $0 \leq i, i_0 < m$ and $\ord_p(\alpha) \geq m$.
  Thus $(a,b,j) \neq (a_0,b_0,j_0)$,
  and in particular $v_{a_0,b_0,j_0}$ remains among the candidates
  $v_1, \ldots, v_n$ considered in Step (4).

  Finally we consider Step (4).
  The application of Algorithm \ref{algo:collisions} is valid,
  because in Step (3) we identified and removed all $v_{a,b,j}$ that are
  equal to one of the $\alpha^i$.
  Moreover, the previous paragraph shows that $v_{a_0,b_0,j_0}$ still remains
  among $v_1, \ldots, v_n$.
  Therefore Algorithm \ref{algo:collisions} will succeed in finding $p$ and $q$.

  If $N$ is prime, then Steps (1)--(4) will of course not find any factors,
  and the algorithm correctly returns ``$N$ is prime'' in Step (5).

  We now analyse the complexity.
  Step (1) performs at most $m$ multiplications in $\ZZ_N$ and computes at most
  $m$ GCDs of integers bounded by $N$, so its cost is
    \[ O(m \MM(\lg N) \lg \lg N) = O(m \lg N \lg^2 \lg N). \]

  The number of pairs $(a,b)$ examined in Step (2) is $O(r \lg r)$.
  For each pair, in Step (2a) the exponent $aN  + b - \lceil (4abN)^{1/2} \rceil$
  lies in $O(N^2)$ (as $a, b \leq r = O(N)$),
  so $t_{a,b}$ may be computed in time $O(\MM(\lg N) \lg N)$.
  Therefore the total cost of computing all of the $t_{a,b}$ is
    \[ O((r \lg r) \MM(\lg N) \lg N) = O(r \lg^3 N \lg \lg N). \]
  For Step (2b), we first precompute $\alpha^{-m} = (\alpha^{m-1} \cdot \alpha)^{-1}$
  in time $O(\MM(\lg N) \lg \lg N)$;
  then for each $j$, computing $v_{a,b,j}$ requires a single multiplication
  by $\alpha^{-m}$, so the cost of deducing all of the $v_{a,b,j}$
  from the $t_{a,b}$ is
    \[ O(s \MM(\lg N)) = O(s \lg N \lg \lg N), \]
  where $s$ is the number of triples $(a,b,j)$ examined altogether.
  We may estimate $s$ as follows:
  \begin{multline*}
  s \leq \sum_{1 \leq ab \leq r} \left(\frac{N^{1/2}}{4rm(ab)^{1/2}} + 1\right)
    = \frac{N^{1/2}}{4rm} \sum_{a=1}^r \frac{1}{a^{1/2}} \sum_{b=1}^{\lfloor r/a \rfloor} \frac{1}{b^{1/2}} + \sum_{a=1}^r \sum_{b=1}^{\lfloor r/a \rfloor} 1 \\
    = O\left(\frac{N^{1/2}}{rm} \sum_{a=1}^r \frac{(r/a)^{1/2}}{a^{1/2}} + \sum_{a=1}^r (r/a) \right)
    = O\left(\frac{N^{1/2}}{r^{1/2} m} \lg r + r \lg r \right),
  \end{multline*}
  so
    \[ s = O\left(\left(\frac{N^{1/2}}{r^{1/2} m} + r\right) \lg N\right). \]

  In Step (3) we construct a list of length $m$,
  consisting of the pairs $(\alpha^i, i)$,
  and a second list of length $s$, consisting of the tuples $(v_{a,b,j},a,b,j)$.
  Each pair/tuple occupies $O(\lg N)$ bits, as $i, a, b, j = O(N)$.
  Using merge sort, we sort the lists by their first component;
  this requires time
    \[ O(s \lg N \lg s + m \lg N \lg m) = O((s + m) \lg^2 N). \]
  Note that the powers $\alpha^i$ in the first list are distinct thanks
  to the hypothesis $\ord_N(\alpha) \geq m$,
  so each $v_{a,b,j}$ is equal to at most one $\alpha^i$.
  Linearly traversing the two lists in parallel,
  we may find all matches $v_{a,b,j} = \alpha^i$ in time
    \[ O((s + m) \lg N). \]
  There are at most $s$ such matches,
  so the cost of invoking Lemma \ref{lem:recover} for all of them is at most
    \[ O(s \lg N \lg \lg N). \]
  Note that the preconditions $a, b, u = O(N)$ for Lemma \ref{lem:recover} hold,
  because $a, b \leq r = O(N)$ and
    \[ u = i + jm + \lceil (4abN)^{1/2} \rceil
         = O\left(m + \frac{N^{1/2}}{4r(ab)^{1/2}} + (rN)^{1/2} \right)
         = O(N). \]

  Finally, in Step (4) we certainly have $n \leq s$,
  so according to Proposition \ref{prop:collisions} the cost of Step (4) is
    \[ O(s \lg^3 N + m \lg^2 N). \]

  Combining the contributions from Steps (1)--(4), the overall complexity is
    \[ O(s \lg^3 N + m \lg^2 N + r \lg^3 N \lg \lg N), \]
  and the claimed bound follows immediately.
\end{proof}

Gluing the various pieces together,
we obtain Algorithm \ref{algo:semiprimes}
(analogous to \cite[Algorithm 6.1]{Hit-timespace}),
which factors any sufficiently large $N$ known to be a prime or semiprime.
\begin{myalgorithm}[ht]
\caption{-- factoring (semi)primes}
\label{algo:semiprimes}
  \begin{flushleft}
  \textbf{Input:} An integer $N \geq 10^9$,
  either prime, or a semiprime $pq$ with $p < q$.

  \textbf{Output:} If $N$ is a semiprime, returns $p$ and $q$. Otherwise returns ``$N$ is prime''.
  \end{flushleft}
  \vspace{2pt} \hrule height 0.2pt \vspace{1pt}
  \begin{enumerate}[label=\arabic*),labelwidth=0.7cm,leftmargin=\labelwidth,itemsep=0.2em]
    \item
    Set
      \[ r \defeq \left\lceil \frac{N^{1/5}}{\lg^{4/5} N} \right\rceil \geq 1, \qquad
	 m \defeq \lceil N^{1/5} \lg^{6/5} N \rceil \geq 1. \]
    
    \item 
    Apply Proposition \ref{prop:small-factors} with
    $M \defeq \lceil (N / r)^{1/2} \rceil \geq 1$.
    If any factors are found, return $p$ and $q$.
    
    \item
    Apply Proposition \ref{prop:find-alpha} with
    $D \defeq \lceil N^{2/5} \rceil \geq 1$.
    If any factors of $N$ are found, or $N$ is proved to be prime, return.
    Otherwise, we obtain $\alpha \in \ZZ_N^*$ such that $\ord_N(\alpha) > D$.

    \item
    Apply Algorithm \ref{algo:search} with the given $N$, $\alpha$, $m$ and $r$.
    Return the factors found, or ``$N$ is prime''.
\end{enumerate}
\end{myalgorithm}

\begin{prop}
\label{prop:semiprimes}
Algorithm \ref{algo:semiprimes} is correct, and its running time is
  \[ O(N^{1/5} \lg^{16/5} N). \]
\end{prop}
\begin{proof}
  In Step (2) we check for prime divisors $p \leq M$.
  Since $M = O(N^{2/5} \lg^{2/5} N)$,
  the cost of invoking Proposition \ref{prop:small-factors} is
    \[ O((N^{2/5} \lg^{2/5} N)^{1/2} \lg^3 N) = O(N^{1/5} \lg^{16/5} N). \]
  After this, we may assume that $N$ is either prime,
  or a semiprime satisfying \eqref{eq:p-condition}.
  
  In Step (3), observe that certainly $N^{2/5} \leq D \leq N$.
  The cost of invoking Proposition \ref{prop:find-alpha} is
  $O(N^{1/5} \lg^2 N)$, which is negligble.
  If Proposition \ref{prop:find-alpha} does not succeed in factoring $N$
  or proving $N$ prime,
  we obtain some $\alpha \in \ZZ_N^*$ with $\ord_N(\alpha) > D$.
  As $N \geq 10^9$ we have $N \geq (\log_2 N + 1)^6 \geq \lg^6 N$,
  so $N^{2/5} \geq N^{1/5} \lg^{6/5} N$ and therefore
  $\ord_N(\alpha) > D \geq m$.
  
  We have verified the preconditions of Algorithm \ref{algo:search},
  so Step (4) correctly returns the factorisation of $N$.
  According to Proposition \ref{prop:search},
  the cost of this step is
    \[ O\left(\left(\frac{N^{1/2}}{r^{1/2} m} + r\right) \lg^4 N + m \lg^2 N \right)
        = O(N^{1/5} \lg^{16/5} N). \qedhere \]
\end{proof}

Finally we may prove the main theorem, following \cite[\S6]{Hit-timespace}.
\begin{proof}[Proof of Theorem \ref{thm:main}]
  We are given as input an integer $N \geq 2$.
  We may reduce to the semiprime case as follows.
  We first apply Proposition \ref{prop:small-factors} to $N$ with
  $M_0 \defeq \lceil N^{1/3} \rceil$.
  If any prime divisor of $N$ is found, we remove it and denote the resulting
  integer by $N_1$.
  We then apply Proposition \ref{prop:small-factors} to $N_1$
  with $M_1 \defeq \lceil N_1^{1/3} \rceil$.
  We repeat this process until no more divisors are found.
  The number of prime divisors of $N$ is $O(\lg N)$,
  and $M_i \leq \lceil N^{1/3} \rceil$ for all $i$,
  so the cost of this procedure is
  $O((N^{1/6} \lg^3 N) \lg N) = N^{1/6+o(1)}$,
  which is negligible compared to $N^{1/5} \lg^{16/5} N$.
  Let $N'$ be the integer remaining at the end of the process.
  Then $1 \leq N' \leq N$ and $N'$ has no more than two prime divisors
  (including multiplicities).
  We may easily check if $N'$ is a square in time $O(\lg N \lg \lg N)$.
  If not, then $N'$ is either a prime or a semiprime.
  If $N' < 10^9$, we may factor $N'$ by any desired algorithm
  (for example trial division).
  Otherwise, we use Algorithm \ref{algo:semiprimes}.
  The desired complexity bound follows from Proposition \ref{prop:semiprimes}.
\end{proof}

\bibliographystyle{amsalpha}
\bibliography{onefifth}

\end{document}